\documentclass[12pt]{amsart}
\usepackage{amssymb,latexsym}
\usepackage[utf8]{inputenc}
\theoremstyle{definition}
\newtheorem{theorem}{Theorem}
\newtheorem{corollary}[theorem]{Corollary}
\newtheorem{proposition}[theorem]{Proposition}
\newtheorem{lemma}[theorem]{Lemma}
\newtheorem{definition}[theorem]{Definition}

\newtheorem{remark}[theorem]{Remark}

\date{}
\newcommand{\numberset}{\mathbb}

\newcommand{\N}{\numberset{N}}

\newcommand{\F}{\numberset{F}}

\newcommand{\Pro}{\numberset{P}}

\newcommand{\Ol}{\mathcal{O}}

\usepackage{hyperref}

\usepackage[margin=2.7cm]{geometry}
\usepackage{amsaddr}

\usepackage{mathptmx}

\begin{document}

\author{Edoardo Ballico$^1$}
\address{Department of Mathematics, University of Trento\\ Via Sommarive 14,
38123 Povo (Trento), Italy}

\author{Alberto Ravagnani$^2$}
\address{Institut de Math\'{e}matiques, Universit\'{e} de Neuch\^{a}tel\\Rue
Emile-Argand 11, CH-2000 Neuch\^{a}tel, Switzerland}
\email{$^1$edoardo.ballico@unitn.it}
\email{$^2$alberto.ravagnani@unine.ch}

\title[]{Projective normality of Artin-Schreier curves}


\maketitle

\providecommand{\bysame}{\leavevmode\hbox to3em{\hrulefill}\thinspace}

\begin{abstract} In this paper we study the projective normality of certain
Artin-Schreier curves $Y_f$ defined over a field $\F$ of characteristic $p$ by
the equations $y^q+y=f(x)$, $q$ being a power of $p$ and $f\in \F[x]$ being a
polynomial in $x$ of degree $m$, with $(m,p)=1$. Many $Y_f$ curves are singular
and so, to be precise, here we study the projective normality of appropriate projective models
of their normalizations.

\end{abstract}

\section{Introduction} \label{intr}
Let $\Pro^2$ denote the projective plane over an arbitrary field $\F$ of
characteristic $p$, and let $q:=p^k$ be a power of $p$ ($k>0$). Denote by $Y_f
\subseteq \Pro^2$ the curve defined over $\F$ by the equation
$y^q+y=f(x)$, where $f(x) \in \F[x]$ is a polynomial of degree $m>0$. Assume
 $(m,p)=1$. The function field $\F(x,y)$ is deeply studied in \cite{Sti},
Proposition 6.4.1. In particular, the function $x$ is known to have only one
pole $P_\infty$. Denote by $\pi:C_f \to Y_f$ the normalization of $Y_f$ (which is known
to be a bijection) and set $Q_\infty:=\pi^{-1}(P_\infty)$. For each $s \ge 0$ the
(pull-backs of the) monomials $x^iy^j$ such that
$$i \ge 0, \ \ 0 \le j \le q-1, \ \ qi+mj \le s$$ form a basis of the vector
space $L(sQ_\infty)$ (see \cite{Sti}, Proposition 6.4.1 again). The genus, $g$, of
the curve $Y_f$ (which is by definition the genus of the normalization $C_f$) is known to be
$g=(m-1)(q-1)/2$. In this paper we study the projective normality of certain
embeddings ($X_f$) of $C_f$ curves into suitable projective spaces. Let us
briefly discuss the outline of the paper.
\begin{itemize}
\item Section \ref{prel} recalls a basic definition and contains a preliminary lemma.
 \item In Section \ref{s1} we take an arbitrary integer $m \ge 2$ which divides $q-1$ and consider the
curve $C_f$ embedded by $L(qQ_\infty)$ into the projective space $\Pro^r$,
$r:=(q-1)/m+1$. We show that this curve is in any case projectively normal and we
compute the dimension of the space of quadric hypersurfaces containing it. 
\item In Section \ref{s2} we pick out an arbitrary integer $m \ge 2$
which divides $tq+1$ ($t$ being any positive integer) and consider the 
curve $C_f$ embedded by $L((tq+1)Q_\infty)$. We 
show that if $(tq-1)/m \le q-1$ and $f(x)=x^m$ then the cited curve is in any case
projectively normal. 
\end{itemize}

Notice that the curve $C_f$ and the line bundles $\mathcal{L}(qQ_\infty)$,
$\mathcal{L}((tq+1)Q_\infty)$ are defined over
any field $\F \supseteq \mathbb {F}_p$ containing the coefficients
of the polynomial $f(x)$. Hence when $f(x) =x^m$ any field of characteristic $p$
may be used.

\section{Preliminaries}
\label{prel}

In this section we recall a basic definition and prove a general
lemma. The result provides in fact sufficient conditions for the projective
 normality of a $C_f$ curve 
as defined in the Introduction.

\begin{definition} \label{defpn}
 A smooth curve $X \subseteq \Pro^r$ defined over a field $\F$ is said to be
\textbf{projectively normal} if for any integer $d \ge 2$  the restriction map
$$\rho_{d,X}:S^d (H^0(\Pro^r, \mathcal{O}_{\Pro^r}(1))) \to H^0(X,\Ol_X(d))$$ is
surjective, $S^d$ denoting the symmetric $d$-power of the tensor product.
\end{definition}

\begin{lemma}\label{trick}
 Consider a $C_f$ curve as in Section \ref{intr} ($q$, $m$ and $f$ being as
in the definitions). Set $C:=C_f$. Fix integers $a,b,e$ such that $a \ge 0$, $b
\ge 0$, $a+b>0$ and $e \ge aq+bm+(m-1)(q-1)-1$. The multiplication map $$\mu:L((aq+bm)Q_\infty)
\otimes L(eQ_\infty) \to L((e+aq+bm)Q_\infty)$$ is surjective.
\end{lemma}
\begin{proof}
 As we will explain below, this is just a particular case of the base-point free
pencil trick (\cite{acgh}, p. 126). Since
the Weierstrass semigroup of non-gaps of $Q_\infty$ contains $m$ and $q$, the
line bundle
$\mathcal {O}_{C}((aq+bm)Q_\infty )$ is spanned by its global sections. Since
$(aq+bm)>0$, we have
$h^0(C,\mathcal {O}_{C}((aq+bm)Q_\infty )) \ge 2$. Hence there is a
two-dimensional linear subspace $V \subseteq H^0(C,\mathcal
{O}_C((aq+bm)Q_\infty))$ (defined over $\overline{\mathbb {F}}$) 
 spanning $\mathcal {O}_{C}((aq+bm)Q_\infty)$. Taking
a basis, say $\{w_1, w_2\}$, of $V$, we get
an exact sequence of line bundles on $C$ (over $\overline{\mathbb {F}}$):
\begin{equation*}
0 \to \mathcal {O}_{C}((e-aq-bm)Q_\infty )\to \mathcal {O}_{C}(eQ_\infty
)^{\oplus 2}\stackrel{\phi}{\to} \mathcal {O}_{C}((e+aq+bm)Q_\infty )\to 0
\end{equation*}
in which $\phi$ is induced by the multiplication by the column vector
$(w_1,w_2)$. By assumption $e-aq -bm >2g-2$. Hence $h^1(C,\mathcal
{O}_{C}((e-aq-bm)Q_\infty ))=0$.
It follows that the map $$\psi : H^0(C,\mathcal {O}_{C}(eQ_\infty ))^{\oplus 2}\to
H^0(C,\mathcal {O}_{C}((e+aq+bm)Q_\infty ))$$ induced in cohomology by the map
$\phi$ of the previous exact sequence is surjective. Since $V\subseteq
H^0(C,\mathcal {O}_{C}((aq+bm)Q_\infty ))$, $\mu$ is surjective.
\end{proof}

In the following sections the previous result will be applied to appropriate
 embeddings of $C_f$ curves.

\section{The case $m|q-1$}
\label{s1}

Assume that $m\ge 2$ is an integer which divides $q-1$ and set $c:=(q-1)/m$. 
If $c=1$ then $Y_f$ is
a smooth plane curve and it is of course projectively normal. Hence we can focus
on the case $c \ge 2$. Notice that the point $P_\infty \in Y_f(\F)$ defined in the Introduction
is the only singular point of $Y_f$, for any choice of
$f(x)$ as in the definitions. We have also an identity of vector spaces
$H^0(C_f, \pi^*(\Ol_{Y_f}(1))) = L(qQ_\infty)$ and by the results stated
in the Introduction it can be easily seen that a basis of them is $\{
1,x,y,...,y^c\}$ ($c \le q-1$ here). Since the linear system spanned by $\{1,x,y\}$
is base-point free (it is also the linear system inducing the composition of
$\pi$ with the inclusion of $Y_f$ in the plane) then also the complete linear
system $L(qQ_\infty)$ is base-point free. Hence it defines a morphism
$\varphi:C_f \to \Pro^r$, $r:=c+1$.

\begin{remark} \label{rmmm}
 Since $\pi$ is injective and has invertible
differential at any point of $C_f \setminus \{ Q_\infty\}$, then also $\varphi$  is
injective with non-zero differential at any point of $C_f \setminus \{
Q_\infty\}$. Moreover, the differential of $\varphi$ is
non-zero even in $Q_\infty$. Indeed, since $L(qQ_\infty)$ has no base-points, 
in order to prove that $\varphi$ has 
 non-zero differential at $Q_\infty$ it is sufficient to
prove that $$h^0(C_f,(q-2)Q_\infty) = h^0(C_f,qQ_\infty )-2$$ 
(\cite{Ha}, Chapter IV, proof of Proposition 3.1). To do this, we may notice
that a basis of $L(qQ_\infty)$ is given 
by a basis of 
$L((q-2)Q_\infty)$ and the monomials $x$ and $y^c$ (see the Introduction). The
result follows.
\end{remark}

By the previous remark,  $\varphi$ is in fact an embedding
of $C_f$ into $\Pro^r$. Set $X_f:=\varphi(C_f)$ and, for any integer $s\ge 1$, denote by
$$\mu_s:L(qQ_\infty) \otimes L(sqQ_\infty) \to L(q(s+1)Q_\infty)$$ the
multiplication map.

\begin{lemma} \label{key1}
 If $\mu_s$ is surjective for all $s \ge 1$ then $X_f$ is projectively normal.
\end{lemma}
\begin{proof}
 Fix an integer $t\ge 2$ and assume that $\mu _s$ is surjective
for all $s\in \{1,\dots ,t-1\}$. We need to prove the surjectivity of
the linear map $\rho _t: S^t(L(qQ_\infty )) \to L(tqQ_\infty )$. Notice
that in arbitrary characteristic $S^t(L(qQ_\infty ))$ is defined as a suitable
quotient
of $L(qQ_\infty )^{\otimes t}$ (\cite{e}, \S A2.3), i.e. $\tau _t = \rho
_t\circ \eta _t$,
where $\tau _t: L(qQ_\infty )^{\otimes t} \to L(tqQ_\infty )$ is the tensor
power map
and $\eta _t: L(qQ_\infty )^{\otimes t} \to S^t(L(qQ_\infty ))$ is a
surjection. Hence $\rho _t$
is surjective if and only if $\tau _t$ is surjective. Since $\tau _2 = \mu _1$,
$\tau _2$ is surjective.
So assume $t>2$ and that $\tau _{t-1}$ is surjective. Since $\tau _{t-1}$ and
$\mu _{t-2}$ are surjective, $\tau _t$ is surjective.
\end{proof}

\begin{proposition} \label{pr_grandi}
 If $s \ge m$ then $\mu_s$ is surjective.
\end{proposition}

\begin{proof}
 If $s \ge m$ then $sq \ge q+(m-1)(q-1)-1$. Apply Lemma \ref{trick} by setting
$e:=sq$, $a:=1$ and $b:=0$.
\end{proof}

\begin{theorem}\label{teo1}
 The curve $X_f$ is projectively normal.
\end{theorem}
\begin{proof}
 By Lemma \ref{key} it is enough to prove that $\mu_s$ is surjective for all $s
\ge 1$. The case $s \ge m$ is covered by Proposition \ref{pr_grandi}. So  let us assume $1
\le s <m$. Let $i,j$ be integers such that $i \ge 0$, $0 \le j \le q-1$ and
$qi+mj \le (s+1)q$.
\begin{itemize}
 \item If $qi+mj \le sq$ then $x^iy^j$ is in the image of $\mu_s$ because $1 \in
L(qQ_\infty)$.
\item If $sq < qi+mj \le (s+1)q$ and $i>0$ then $x^{i-1}y^j \in L(sqQ_\infty)$.
Since $x \in L(qQ_\infty)$ then $x^iy^j$ is in the image of $\mu_s$.
\item If $i=0$ and $sq < mj < (s+1)q$ then $j>sq/m=s(c+1/m)>c$ and $mj \le
(s+1)q-1$. By the latter inequality we get $m(j-c) \le (s+1)q-1-mc$. Observe
that $(s+1)q-1-mc = sq$ and so $m(j-c) \le sq$. This proves that $y^{j-c} \in
L(sqQ_\infty)$. Finally, $\mu_s(y^c \otimes y^{j-c})=y^j$.
\item If $i=0$ and $mj=(s+1)q$ then $j=(s+1)q/m=(s+1)(c+1/m)=(s+1)c+(s+1)/m$.
Since $0 \le j \le q-1$ is a nonnegative integer, we must have $(s+1)/m \in \N$.
Since $1 \le s <m$ we get $s=m-1$. It follows $mj=mq$ and $j=q$, a
contradiction.
\end{itemize}
This proves the theorem.
\end{proof}

\begin{corollary}\label{coro1}
Assume $m \ge 3$. The curve $X_f\subseteq \Pro^{r}$ is contained into
$\binom{c+3}{2}-3c-3$ linearly independent quadric hypersurfaces.
\end{corollary}

\begin{proof}
 In the notations of Definition \ref{defpn} set $X:=X_f$ and $d:=2$. Define
$r:=c+1$. By Theorem \ref{teo1} the restriction map 
$$\rho_{2,X_f}: S^2(H^0(\Pro^r,\Ol_{\Pro^r}(1))) \to H^0(X_f, \Ol_{X_f}(2))$$ is
surjective. Hence, in particular, the restriction map

$$\rho: H^0(\Pro^r,\Ol_{\Pro^r}(2)) \to H^0(X_f, \Ol_{X_f}(2))$$ is surjective.
Since $m \ge 3$ (by assumption) we can easily check that a basis of the vector
space $L(2qQ_\infty)$ consists of the following monomials:  
$$\{ 1,y,...,y^{2c},x,xy,...,xy^{c},x^2 \}.$$
Hence $h^0(X_f,\Ol_{X_f}(2))=\dim_{\F} \ L(2(q-1)Q_\infty)=3c+3$. The kernel of
$\rho$ is exactly the space of the quadrics in $\Pro^r$ vanishing on $X_f$. By
the surjectivity of $\rho$ we easily deduce its dimension:
$$\dim_{\F}
H^0(\Pro^r,\mathcal{I}_{X_f}(d))=\binom{r+2}{2}-(3c+3)=\binom{c+3}{2}-3c-3.$$
The result follows.
\end{proof}

\section{The case $m|tq+1$}\label{s2}
Pick out any integer $t\ge 1$ and assume that $m \ge 2$ is an integer dividing $tq+1$.
Set $c:=(tq+1)/m$. As in Remark \ref{rmmm}, it can be checked
that $L((tq+1)Q_\infty)$ defines an embedding, say $\varphi$, of $C_f$ into
$\Pro^r$, $r:=\dim L((tq+1)Q_\infty)-1$. Define $X_f:=\varphi(C_f)$. For any 
integer $s\ge 1$ denote by
$$\mu_s:L((tq+1)Q_\infty) \otimes L(s(tq+1)Q_\infty) \to
L((s+1)(tq+1)Q_\infty)$$ the multiplication map. As in Section \ref{s1}, the
projective normality of $X_f$ is controlled by the $\mu_s$ maps.

\begin{lemma} \label{key}
 If $\mu_s$ is surjective for all $s \ge 1$ then $X_f$ is projectively normal.
\end{lemma}

\begin{proof} Take the proof of Lemma \ref{key1}.  \end{proof}

\begin{proposition}\label{grandi}
 If $s \ge m$ then $\mu_s$ is surjective.
\end{proposition}

\begin{proof}
 Apply Lemma \ref{trick} by setting $a:=0$, $b:=c$ and $e:=s(tq+1)$.
\end{proof}

In the following part of the section we focus on the case $f(x)=x^m$. In
particular we are going to show that $X_f$ curves obtained with this choice
of $f$ are projectively normal for any choice of $t \ge 1$, provided that $c \le
q-1$.

\begin{remark}
 The assumption $c \le q-1$ is not so restrictive from a geometric point of
view. In fact, for any fixed $q$, the genus of $X_f$ is $g=(q-1)(m-1)/2$. Even
if $c$ is small, here we study many curves of interesting genus. 
\end{remark}

\begin{lemma} \label{y^b}
Set $f(x):=x^m$ and assume $c \le q-1$. Pick out an integer $b \ge 0$
 such that $b \le (s+1)c$.
Then $y^b$ is in the image of $\mu_s$.
\end{lemma}
\begin{proof}
 Since $c \le q-1$ we get $y^c \in L((tq+1)Q_\infty)$. In particular, if $b \le
c$ then we are done. Assume $b>c$.
 Let us prove the lemma by induction on $s$. If $s=1$, then $b \le 2c$ and $b-c
\le c \le q-1$.
Hence $y^{b-c} \in L((tq+1)Q_\infty)$ and so $y^b=\mu_1(y^c \otimes y^{b-c})$ is of course
in the image of $\mu_1$.
If $s>1$, then 
write $b=hc+\rho$ with $h \le s$ and $0 \le \rho \le c$. Since $b-\rho=hc \le sc$ we have
that $y^{b-\rho}$
is in the image of $\mu_{s-1}$. In particular, it is in $L(s(tq+1)Q_\infty)$.
Since $y^\rho \in L((tq+1)Q_\infty)$, we get
$y^b=\mu_s(y^\rho \otimes y^{b-\rho})$. It follows that $y^b$ is in the image of $\mu_s$.
\end{proof}

\begin{theorem}\label{te2}
 Set $f(x)=x^m$ and assume  $c \le q-1$. Then $X_f$ is projectively normal.
\end{theorem}

\begin{proof}
By Lemma \ref{key}, it is enough to show that $\mu_s$ is surjective for any $s
\ge 1$. By Proposition \ref{grandi}, we need only to prove that $\mu_s$ is
surjective for any $1 \le s<m$. Let $i,j$ be integers such that $i \ge 0$, $0 \le j \le
q-1$ and $qi+mj \le (s+1)(tq+1)$.
We will examine separately the case $2 \le s <m$ and the case $s=1$. 

To begin with, assume $2 \le s < m$.

\begin{itemize}
 \item If $j \ge c$, then $x^iy^{j-c} \in L(s(tq+1)Q_\infty)$. Since $y^c \in
L((tq+1)Q_\infty)$, we have $x^iy^j=\mu_s(y^c \otimes x^iy^{j-c})$.
\item If $0 \le j < c$ and $qi+mj \le s(tq+1)$, then $x^iy^j$ is in the image of
$\mu_s$, because $1 \in L((tq+1)Q_\infty)$.
\item Assume $0 \le j<c$ and $s(tq+1)<qi+mj \le (s+1)(tq+1)$. We have $i \ge t$.
Indeed, assume by contradiction that $i <t$. Then 

\begin{eqnarray*}
 qi+mj &<&  tq+mj \\ 
&<& tq+mc \\ &=& tq+tq+1 \\ &\le& stq+1  \\ &<&  s(tq+1), 
\end{eqnarray*}
a contradiction (here we used $s \ge 2$).
\begin{itemize}
\item[(A)] If $qi+mj < (s+1)(tq+1)$ then $x^{i-t}y^j \in L(s(tq+1)Q_\infty)$ and
$x^iy^j= \mu_s(x^t\otimes x^{i-t}y^{j})$.
\item[(B)] Assume $qi+mj=(s+1)(tq+1)$. Since $(m,p)=1$ we have $i=am$ for an
integer $a>0$ and $j= (s+1)c-aq$. Observe that
$x^iy^j=x^{am}y^{(s+1)c-aq}={(y^q+y)}^ay^{(s+1)c-aq}$, which is a sum of
monomials of the form $y^b$ with $b \le (s+1)c$. Apply Lemma \ref{y^b} and the
fact that $\mu_s$ is linear to get that $x^iy^j$ is in its image.
\end{itemize}
\end{itemize}

Now assume $s=1$.
\begin{itemize}
 \item Assume $j \ge c$. Since $qi+mj \le 2(tq+1)$ we get $qi+m(j-c) \le
2(tq+1)-(tq+1)=tq+1$. Hence $x^iy^{j-c} \in L((tq+1)Q_\infty)$. Finally,
$\mu_1(y^c \otimes x^iy^{j-c})=x^iy^j$.
\item Assume $j<c$ and $i \ge t$. Since $qi+mj \le 2(tq+1)$ we get $q(i-t)+mj
\le 2(tq+1)-tq=tq+2$.
\begin{itemize}
 \item[(C)] If $q(i-t)+mj \le tq+1$ then $x^iy^j=\mu_1(x^t\otimes x^{i-t}y^j)$.
\item[(D)] Assume $q(i-t)+mj=tq+2$, i.e. $qi+mj=2tq+2$. Repeat the proof of case
(B) with $s:=1$.
\end{itemize}
\item Assume $j<c$ and $i<t$. Then $x^i, y^j \in L((tq+1)Q_\infty)$ and we
easily get $x^iy^j = \mu_1(x^i\otimes y^j)$.
\end{itemize}
The proof is concluded.
\end{proof}

\begin{remark}
 If $t=1$ then the assumption $c \le q-1$ is trivially satisfied (we assumed $m
\neq q+1$). In this case the curve $y^q+y=x^m$ is covered by the Hermitian curve.
\end{remark}

\section*{Acknowledgment}
The authors would like to thank the Referee for suggestions and remarks which improved the presentation
of the present work.

\end{document}